\documentclass[preprint]{elsarticle}

\makeatletter
\def\ps@pprintTitle{%
	\let\@oddhead\@empty
	\let\@evenhead\@empty
	\def\@oddfoot{\footnotesize\itshape
		{} \hfill\today}%
	\let\@evenfoot\@oddfoot
}
\makeatother

\usepackage{latexsym}
\usepackage{lineno}
\usepackage{hyperref}
\usepackage{indentfirst}
\usepackage{amsxtra}
\usepackage{amssymb}
\usepackage{amsthm}
\usepackage{natbib}
\usepackage{amsmath}
\usepackage{tikz}
\usepackage{amscd}
\usepackage{MnSymbol}

\usepackage[capitalise]{cleveref}
\newtheorem{theor}{Theorem}
\newtheorem{prop}[theor]{Proposition}
\newtheorem{cor}[theor]{Corollary}
\newtheorem{lemma}[theor]{Lemma}

\theoremstyle{definition} 
\newtheorem{defin}{Definition}
\newtheorem{rem}{Remark}
\newtheorem{ques}{Question}

\newtheorem*{conv}{Convention}
\newtheorem{ex}[theor]{Example}
\newtheorem{exs}[theor]{Examples}

\DeclareMathOperator{\Sym}{Sym}
\DeclareMathOperator{\id}{id}
\DeclareMathOperator{\Aut}{Aut}

\begin{document}

\begin{frontmatter}
	\title{On the indecomposable involutive solutions of the Yang-Baxter equation of finite primitive level 
 }
	\tnotetext[mytitlenote]{The author is a member of GNSAGA (INdAM).}
	\author{Marco CASTELLI}
	\ead{marco.castelli@unisalento.it - marcolmc88@gmail.com}
	%

\begin{abstract}
In this paper, we study the class of indecomposable involutive solutions of the Yang-Baxter equation of finite primitive level, recently introduced by Ced\'o and Okni\'nski in \cite{cedo2021constructing}. We give a group-theoretic characterization of these solutions by means of displacements groups and we apply this result to compute and enumerate the ones having small size. For some classes of indecomposable involutive solutions recently studied in literature, we compute the exact value of the primitive level. Some relationships with other families of solutions also are discussed. Finally, following \cite[Question 3.2]{cedo2021constructing}, we completely describe the ones having primitive level $2$ by left braces. 
\end{abstract}
\begin{keyword}
\texttt{imprimitive group\sep Yang-Baxter equation\sep brace\sep cycle set }
\MSC[2020] 16T25\sep 81R50 \sep 20E22 \sep 20N02
\end{keyword}

\end{frontmatter}

\section*{Introduction}
The quantum Yang-Baxter equation have been of interest ever since a paper of Yang \cite{yang1967}, where it appears the first time. Given a vector space $V$, a map $R:V\otimes V\rightarrow V\otimes V$ is said to be a \textit{solution of the quantum Yang-Baxter equation} if 
$$ 
R_{12}R_{13}R_{23}=R_{23}R_{13}R_{12}
$$
where $R_{ij}:V\otimes V\otimes V\rightarrow V\otimes V\otimes V$ is the map acting as $R$ on the $(i,j)$ tensor factor and as the identity on the remaining factor. Finding all the solutions of the quantum Yang-Baxter equation seems to be very hard and it is still an open problem. In that regard Drinfeld \cite{drinfeld1992some} suggested the study of an easier case, i.e., the solutions of the quantum Yang-Baxter equation that are induced by the linear extension of a map $\mathcal{R}:X\times X\rightarrow X\times X $, where $X$ is a basis for $V$. A function $\mathcal{R}$ of this type is called a \textit{set-theoretic solution of the quantum Yang-Baxter equation}. In the last years, several authors studied these solutions using an equivalent formulation. Specifically, a map $r:X\times X\to X\times X$ is said to be a \emph{set-theoretic solution of the Yang-Baxter equation} if
$$
r_1r_2r_1 = r_2r_1r_2 ,
$$
where $r_1:= r\times id_X$ and $r_2:= id_X\times r$. It is easy to see that if $\tau:X\times X\rightarrow X\times X$ is the twist map, then a function $\mathcal{R}:X\times X\rightarrow X\times X $ is a set-theoretic solution of the quantum Yang-Baxter equation if and only if the map $r:=\tau\mathcal{R}$ is a set-theoretic solution of the Yang-Baxter equation. 
Now, let $\lambda_x:X\to X$ and $\rho_y:X\to X$ be maps such that $r(x,y) = (\lambda_x(y), \rho_y(x))$ for all $x,y\in X$. A set-theoretic solution of the Yang-Baxter equation $(X, r)$, which we will simply call \emph{solution}, is said to be a left [\,right\,] non-degenerate if $\lambda_x\in Sym(X)$ [\,$\rho_x\in Sym(X)$\,] for every $x\in X$ and \textit{non-degenerate} if it is left and right non-degenerate. By seminal papers of Etingof, Schedler and Soloviev \cite{etingof1998set} and Gateva-Ivanova and Van den Bergh \cite{gateva1998semigroups} the involutive solutions, i.e., the ones such that $r^2=id_X$, have received a lot of attention. 
In this context, various methods to construct new involutive set-theoretic solutions were provided (see for example \cite{rump2007braces,vendramin2016extensions}). A first attempt, made in \cite{etingof1998set}, is based on the notion of \emph{retraction}: starting from a solution $(X,r)$, it allows to construct a new involutive solution, indicated by $Ret(X,r)$, identifying two elements $x,y$ whenever $\lambda_x=\lambda_y$. If the retraction of an involutive solution $(X,r)$ does not provide a new solution, i.e. if $(X,r)=Ret(X,r)$, then $(X,r)$ is called \emph{irretractable}; while if the retraction-process of $(X,r)$ stabilizes to a singleton, then the solution is called \emph{multipermutation} and the number of retractions-iterations is called \emph{multipermutation level}. Roughly speaking, the multipermutation level measures how far  a solution is from being \emph{trivial}, i.e. a solution for which $\lambda_x=\lambda_y$ for all $x,y\in X$. Particular attention was devoted to the class of \emph{indecomposable} involutive solutions, since every indecomposable involutive solution can be constructed by dynamical extension of a simple solution (see \cite[Proposition 2]{cacsp2018}) and, moreover, these solutions carry informations on every involutive solution not necessarily indecomposable, as all the involutive solutions are constructed by the indecomposable ones (see \cite[Section 2]{etingof1998set}). A successful strategy consist of studying these solutions by associating them various algebraic structures, such as cycle sets, biracks, structure monoids and groups (see for example \cite{chouraqui2010garside,JePiZa20x,rump2020}). In that regard, in 2007 Rump \cite{rump2007braces} introduced an algebraic structure called \textit{left brace}. Recall that a left brace is a set $A$ with two operations $+$ and $\circ$ such that $(A,+)$ is an abelian group, $(A,\circ)$ is a group and 
$$a\circ (b+c)+a=a\circ b+a \circ c$$
for every $a,b,c \in A$.  As showed in \cite[Section 1]{rump2007braces}, left braces provide involutive solutions. Conversely, every involutive solution can be constructed by a left brace (see \cite{bachiller2016solutions} for more details) for which the multiplicative group coincides with a standard permutation group called the \emph{associated permutation group}. In particular, an arbitrary indecomposable solution $(X,r)$ can be recovered by a suitable left brace $B$ and a core-free subgroup $H$ of $(B,\circ)$, identifying $X$ with the left cosets $B/H$.\\
In recent years, left braces are used is a systematic way to give structural results on indecomposable multipermutation solutions, as made for example in \cite{cedo2020primitive,jedlivcka2021cocyclic,rump2020one}. 
Much less is known about indecomposable solutions that are not multipermutation: the only remarkable results are recently given in \cite{castelli2022characterization,cedo2021constructing,cedo2022new} for the family of the simple ones. Moreover, up to now there is not an analogue of the multipermutation level to measure in some sense the complexity of an indecomposable solution that is not a multipermutation solution. To partially fill these discrepancies, in this paper we change point of view studying the notion of indecomposable solution of \emph{finite primitive level}. It was introduced the first time in \cite{cedo2021constructing}, after the main result of \cite{cedo2020primitive} in which it was shown that, apart the indecomposable involutive solutions of prime size, all the indecomposable involutive solutions have imprimitive associated permutation group. This implicitly suggested a new perspective to study indecomposable solutions, focusing on the imprimitive blocks systems.  Actually, the family of indecomposable involutive solutions of finite primitive level is an unexplored topic, except for a recent paper of the author \cite{castelli2022simplicity}, where a generalization on the non-involutive case is also given. The first main result of the paper provides a group-theoretic characterization of indecomposable solutions of finite primitive level by means of a standard subgroup of the associated permutation group called the \emph{displacements group}, already considered in \cite{jedlivcka2023indecomposable} to study multipermutation solutions and in \cite{bon2019} to study latin solutions. We remark that this result, even if is proved by means of left braces and cycle sets, allow to detect all the indecomposable involutive solutions of finite primitive level only focusing on the action of the displacements groups. As an application of this fact, we find all the indecomposable involutive solutions of finite primitive level among the ones having size $\leq 9$: we summarize the computation in Section $3$. Here, we also exhibit several examples of indecomposable involutive solutions of finite primitive level and we compute the value of the primitive level for some families of solutions. In this context, the links with other classes of solutions, such as latin solutions and soluble solutions, are also discussed. Moreover, we consider the cycles decomposition of the maps $\lambda_x$ of these solutions and we take advantage of our result to give a partial answer to \cite[Question 3.16]{RaVe21}, providing a decomposability criterion for multipermutation solutions.
Similarly to the approach used for multipermutation solution, for which several authors provided nice description of the ones having low multipermutation level (see for example \cite{capiru2020,JePiZa20x}), in the last part of the paper we focus on indecomposable solution having primitive level $2$. In particular, following \cite[Question 3.2]{cedo2021constructing}, we provide a left brace-theoretic description of these solution, and we illustrate it by an example.

\section{Basic definitions and results}

In this section we give the preliminaries involving cycle sets and braces useful through the paper. 

\subsection{Solutions of the Yang-Baxter equation and cycle sets}
In \cite{rump2005decomposition}, Rump found a one-to-one correspondence between solutions and an algebraic structure with a single binary operation, which he called \emph{non-degenerate cycle sets}.
To illustrate this correspondence, let us firstly recall the following definition.

\begin{defin}[pag. 45, \cite{rump2005decomposition}]
A pair $(X,\cdot)$ is said to be a \emph{cycle set} if each left multiplication $\sigma_x:X\longrightarrow X,$ $y\mapsto x\cdot y$ is bijective and 
$$(x\cdot y)\cdot (x\cdot z)=(y\cdot x)\cdot (y\cdot z), $$
for all $x,y,z\in X$.  Moreover, a cycle set $(X,\cdot)$ is called \textit{non-degenerate} if the squaring map $\mathfrak{q}:X\longrightarrow X$, $x\mapsto x\cdot x$ is bijective.
\end{defin}

\begin{ex}
    The simplest example of cycle set is the one given by $X:=\mathbb{Z}/n\mathbb{Z}$ and $x\cdotp y:=\alpha(y)$, for an arbitrary number $n$ and a permutation $\alpha$ of $Sym(X)$. We will call these cycle sets \emph{trivial}.
\end{ex}

Cycle sets are useful to construct new solutions of the Yang-Baxter equation.

\begin{prop}[Propositions 1-2, \cite{rump2005decomposition}]\label{corrisp}
Let $(X,\cdot)$ be a cycle set. Then the pair $(X,r)$, where $r(x,y):=(\sigma_x^{-1}(y),\sigma_x^{-1}(y)\cdot x)$, for all $x,y\in X$, is an involutive left non-degenerate solution of the Yang-Baxter equation which we call the associated solution to $(X,\cdot)$. Moreover, this correspondence is one-to-one.
\end{prop}

\begin{conv}
    From now on, every cycle set will be non-degenerate. All the results will be given in the language of cycle sets, but they can be translated by \cref{corrisp}.
\end{conv}

An useful tool to construct new cycle sets from a given one, introduced in \cite{etingof1998set}, is the so-called \textit{retract relation}.
Specifically, in \cite{rump2005decomposition} Rump showed that the binary relation $\sim_\sigma$ on $X$ given by 
$$x\sim_\sigma y :\Longleftrightarrow \sigma_x = \sigma_y$$ 
for all $x,y\in X$, is a \emph{congruence} of $(X,\cdot)$, i.e. an equivalence relation for which $x\sim_\sigma y$ and $x'\sim_\sigma y'$ implies $x\cdotp x'\sim_\sigma y\cdotp y',$ for all $x,x',y,y'\in X$. In \cite{etingof1998set} (and indipendently in \cite{rump2005decomposition}) it was showed that the quotient $X/\sim_{\sigma}$, which we denote by $Ret(X)$, is a cycle set, that we will call \emph{retraction} of $(X,\cdot)$. 
An important class of cycle sets is given by the ones having finite multipermutation level.

\begin{defin}
    A cycle set $X$ is said to be of \emph{multipermutation level $n$} if $n$ is the minimal non-negative integer such that $|Ret^n(X)|=1$, where $Ret^n(X) $ is the cycle set defined inductively by $Ret^0(X)=X$ and $Ret^n(X)=Ret(Ret^{n-1}(X))$, for all $n\in \mathbb{N}$.
\end{defin}

In a classical way we can define the notion of cycle sets homomorphism.

\begin{defin}
    Let $X,Y$ be cycle sets. A map $p:X\longrightarrow Y$ is said to be a \emph{homomorphism} between $X$ and $Y$ if $p(x\cdotp y)=p(x)\cdotp p(y)$ for all $x,y\in X$. A surjective homomorphism is called \emph{epimorphism}, while a bijective homomorphism is said to be an \emph{isomorphism}.
\end{defin}

Two standard permutation groups related to a cycle sets $X$ are the one generated by the set $\{ \sigma_x|\hspace{1mm} x\in X\}$, called the \emph{associated permutation group of $X$} and indicated by $\mathcal{G}(X)$, and the one generated by the set $\{ \sigma_x\sigma_y^{-1}|\hspace{1mm} x,y\in X\}$, called the \emph{displacements group of $X$} and indicated by $Dis(X)$.\\
In this context, our attention will be focused on indecomposable cycle sets.

\begin{defin}
A cycle set $(X,\cdot)$ is said to be \textit{indecomposable} if the permutation group $\mathcal{G}(X)$ acts transitively on $X$. 
\end{defin}

\smallskip

Following the paper by Vendramin \cite{vendramin2016extensions}, if $X$ is a cycle set, $S$ a set and $\alpha:X\times X\times S\longrightarrow \Sym(S)$, $\alpha(x,y,s)\mapsto\alpha_{(x,y)}(s,-)$ a function such that
\begin{equation}\label{cociclo}
\alpha_{(x\cdot y,x\cdot z)}(\alpha_{(x,y)}(r,s),\alpha_{(x,z)}(r,t))=\alpha_{(y\cdot x,y\cdot z)}(\alpha_{(y,x)}(s,r),\alpha_{(y,z)}(s,t)),
\end{equation}
for all $x,y,z\in X$ and $r,s,t\in S$, then $\alpha$ is said to be a \textit{dynamical cocycle} and the operation $\cdot$ given by 
$$
(x,s)\cdot (y,t):=(x\cdot y,\alpha_{(x,y)}(s,t)),
$$
for all $x,y\in X$ and $s,t\in S$ makes $X\times S$ into a cycle set which we denote by $X\times_{\alpha} S$ and we call \textit{dynamical extension} of $X$ by $\alpha$. 
A dynamical extension $X\times_{\alpha} S$ is called \textit{indecomposable} if $X\times_{\alpha} S$ is an indecomposable cycle set: by \cite[Theorem 7]{cacsp2018}, this happens if and only if $X$ is an indecomposable cycle set and the subgroup of $\mathcal{G}(X\times S)$ generated by $\{h \ | \ \forall \, s\in S \quad h(y,s)\in \{y\}\times S \}$ acts transitively on $\{y\}\times S$, for some $y\in X$. By results contained in \cite{cacsp2018,vendramin2016extensions}, the following corollary, which is of crucial importance for this paper, follows.

\begin{cor}[Theorem 7 of \cite{cacsp2018} and Theorem 2.12 of \cite{vendramin2016extensions}]\label{dyn}
    Let $X$ be an indecomposable cycle set, $Y$ a cycle set and $p:X\longrightarrow Y $ a cycle set epimorphism. Then, there exist a set $S$ and a dynamical cocycle $\alpha$ such that $X$ is isomorphic to $Y\times_{\alpha} S $.
\end{cor}

\subsection{Indecomposable cycle sets and left braces}
At first, we introduce the following definition that, as observed in \cite{cedo2014braces}, is equivalent to the original introduced by Rump in \cite{rump2007braces}.

\begin{defin}[\cite{cedo2014braces}, Definition 1]
A set $B$ endowed of two operations $+$ and $\circ$ is said to be a \textit{left brace} if $(B,+)$ is an abelian group, $(B,\circ)$ a group, and
$$
    a\circ (b + c) + a
    = a\circ b + a\circ c,
$$
for all $a,b,c\in B$.
\end{defin}

\begin{exs}\label{esbrace}
\begin{itemize}
    \item[1)] If $X$ is a cycle set, then one can show that the free abelian group $\mathbb{Z}^X$ gives rise to a left brace $(\mathbb{Z}^X,+,\circ)$, where $(\mathbb{Z}^X,\circ)$ is the group having $X$ has generating set and $x\circ y=\sigma^{-1}_x(y)\circ (\sigma^{-1}_x(y)\cdotp x)$, where $x,y\in X$, as relations.
    \item[2)] If $X$ is a cycle set, the associated permutation group $\mathcal{G}(X)$ gives rise to a left brace $(\mathcal{G}(X),+,\circ)$, where $\circ$ is the usual composition in $\mathcal{G}(X)$ (see, for example, \cite[Section 2]{bachiller2015family} for more details). From now on, we will refer to $(\mathcal{G}(X),+,\circ)$ as the \emph{permutation left brace}.
    \item[3)] If $(B,+)$ is an abelian group, then the operation $\circ$ given by $a\circ b:=a+b$ give rise to a left brace which we will call \emph{trivial} 
\end{itemize}

\end{exs}

If $(B_1,+,\circ)$ and $(B_2,+',\circ')$ are left braces, a homomorphism $\psi$ between $B_1$ and $B_2$ is a function from $B_1$ to $B_2$ such that $\psi(a+b)=\psi(a)+'\psi(b)$ and $\psi(a\circ b)=\psi(a)\circ'\psi(b)$, for all $a,b\in B_1$.

\smallskip

Given a left brace $B$ and $a\in B$, let us denote by $\lambda_a:B\longrightarrow B$ the map from $B$ into itself defined by
\begin{equation*}\label{eq:gamma}
    \lambda_a(b):= - a + a\circ b,
\end{equation*} 
for all $b\in B$. 
As shown in \cite[Proposition 2]{rump2007braces} and \cite[Lemma 1]{cedo2014braces}, these maps have special properties. We recall them in the following proposition.
\begin{prop}\label{action}
Let $B$ be a left brace. Then, the following are satisfied: 
\begin{itemize}
\item[1)] $\lambda_a\in\Aut(B,+)$, for every $a\in B$;
\item[2)] the map $\lambda:B\longrightarrow \Aut(B,+)$, $a\mapsto \lambda_a$ is a group homomorphism from $(B,\circ)$ into $\Aut(B,+)$.
\end{itemize}
\end{prop}



For the following definition, we refer the reader to \cite[pg. 160]{rump2007braces} and \cite[Definition 3]{cedo2014braces}.

\begin{defin}
Let $B$ be a left brace. A subset $I$ of $B$ is said to be a \textit{left ideal} if it is a subgroup of the multiplicative group and $\lambda_a(I)\subseteq I$, for every $a\in B$. Moreover, a left ideal is an \textit{ideal} if it is a normal subgroup of the multiplicative group.
\end{defin}
%
\noindent As one can expect, if $I$ is an ideal of a left brace $B$, then the structure $B/I$ is a left brace called the \emph{quotient left brace} of $B$ modulo $I$. Moreover, the ideal $\{0\}$ will be called the \emph{trivial} ideal and and a left brace $B$ which contains no ideals different from $\{0\}$ and $B$ will be called a \emph{simple} left brace.\\
A standard ideal of a left brace $B$, given in \cite[Corollary of Proposition 6]{rump2007braces} and indicated by $B^2$, is the one given by the additive subgroup generated by the set $\{a*b\hspace{1mm} |\hspace{1mm} a,b\in B \}$, where $a*b:=-a+a\circ b-b$ for all $a,b\in B$.\\
Other ideals can be obtained by left braces homomorphisms. Indeed, if $B_1$ and $B_2$ are left braces and $\psi$ a homomorphism from $B_1$ to $B_2$, the kernel of $\psi$ is an ideal of $B_1$, where the kernel, which we indicate by $Ker(\psi)$, is the set given by $Ker(\psi):=\{b\in B_1|\psi(b)=0\}$.\\
In \cite{rump2007braces}, Rump also introduced an other ideal that, in the terms of \cite[Section 4]{cedo2014braces}, is the following.
\begin{defin}
Let $B$ be a left brace. Then, the set 
$$
    Soc(B) := \{a\in A \ | \ \forall \,
     b\in B \quad  a + b = a\circ b \}
$$
is named \emph{socle} of $B$.
\end{defin}
\noindent Clearly,
$Soc(B) := \{a\in B \ | \ \lambda_a = \id_B\}$.
Moreover, we have that $Soc(B)$ is an ideal of $B$. The two left braces given in \cref{esbrace} are related by the socle. Indeed, given a cycle set $X$, one can show that the map $\theta:X\longrightarrow \mathcal{G}(X)$ given by $x\mapsto \sigma^{-1}_x $ can be extended to a surijective left brace homomorphism $\bar{\theta}:\mathbb{Z}^X\longrightarrow \mathcal{G}(X)$ such that $Ker(\bar{\theta})=Soc(\mathbb{Z}^X)$.  




\smallskip

Left braces homomorphisms are strongly related to cycle sets homomorphisms. Indeed, every cycle set epimorphism $p:X\longrightarrow Y$ induces a left brace epimorphisms $p':\mathbb{Z}^X\longrightarrow \mathbb{Z}^Y$ by $x\mapsto p(x)$ for all $x\in X$ and $\bar{p}:\mathcal{G}(X)\longrightarrow \mathcal{G}(Y)$ by $\sigma_x\mapsto \sigma_{p(x)}$, for all $x\in X$. 

\begin{defin}
    A cycle sets epimorphism $p:X\longrightarrow Y$ is said to be a \emph{covering} if induces a left braces isomorphism $\bar{p}:\mathcal{G}(X)\longrightarrow \mathcal{G}(Y)$. 
\end{defin}


In the last part of the section, we recall the theory mainly developed in \cite{rump2023primes, rump2020}, that allow to detect more informations on indecomposable cycle sets (and their epimorphic images) by left braces.

\begin{prop}[Theorem 3, \cite{rump2020}]\label{cosmod}
Let $(B,+,\circ)$ be a left brace, $Y\subset B$ a transitive cycle base, $a_1\in Y$,
and $K$ a core-free subgroup of $(B,\circ)$, contained in the
stabilizer  $B_{a_1}$ of $a_1$ (respect to the action $\lambda$). Then, the pair $(X,\cdotp)$ given by $X:=B/K$ and $\sigma_{x\circ K}(y\circ K):=\lambda_x(a_{1})^-\circ y \circ K$ give rise to an indecomposable cycle set with $ \mathcal{G}(X)\cong B$.\\
 Conversely, every indecomposable cycle set $(X,\cdotp)$ with $\mathcal{G}(X) \cong B $ (as left braces) can be obtained in this way.
\end{prop}

\begin{conv}
    From now on, a cycle set obtained as in  \cref{cosmod} will be indicated by $C_{B,K,a_1}$.
\end{conv}

\begin{prop}\label{coveringcostr}
Let $p:X\longrightarrow Y$ a convering of indecomposable cycle sets, $\mathcal{G}:= \mathcal{G}(X)$ and $x\in X$. Then, there exist two core-free subgroups $K$ and $H$ contained in $\mathcal{G}_{\sigma_x}$ with $K\leq H$ such that $X$ can be identified with $C_{\mathcal{G},K,\sigma_x}$, $Y$ can be identified with $C_{\mathcal{G},H,\sigma_x}$ and $p$ can be identified with the epimorphism from $C_{\mathcal{G},K,\sigma_x}$ to $C_{\mathcal{G},H,\sigma_x}$ which sends an element $z\circ K$ to $z\circ H$.\\
Conversely, up to isomorphisms, any covering of indecomposable cycle sets arises in this way.
\end{prop}

\begin{proof}
    It follows by \cite[Theorem 3 and Corollary 2]{rump2020}.
\end{proof}

\begin{prop}[Section 3, \cite{rump2023primes}]\label{epiid}
     Let $X,Y$ be indecomposable cycle sets and $p:X\longrightarrow Y$ an empimorphism from $X$ to $Y$. Then, the set $I:=\{g|g\in \mathcal{G}(X), \quad p(g(x))=p(x)\quad \forall \quad x\in X \}$ is an ideal of $\mathcal{G}(X)$.\\
     Conversely, if $X$ is an indecomposable cycle set and $I$ an ideal of $\mathcal{G}(X)$, the $I$-orbits of $X$ induces a cycle set structure, which we indicate by $X/I$, that give rise to a canonical cycle set epimorphism $p:X\longrightarrow X/I$.
\end{prop}

From now on, if $X,Y$ are indecomposable cycle sets and $p:X\longrightarrow Y$ an empimorphism from $X$ to $Y$ the ideal induced by $p$ as in the previous result will be indicated by $I(p)$. If $X$ is an indecomposable cycle set and $I$ an ideal of $\mathcal{G}(X)$, the induced epimorphism from $X$ to $X/I$ will be indicated by $p_I$. 

\begin{prop}[Theorem 1, \cite{rump2023primes}]\label{fattepi}
    Let $X,Y$ be indecomposable cycle sets and $p:X\longrightarrow Y$ an epimorphism from $X$ to $Y$. Up to isomorphism, there exist a unique factorization $p=qp_I$ for a suitable ideal $I$ of $\mathcal{G}(X)$, where $q$ is a covering of cycle sets. In particular, $Y$ is an epimorphic image of $X/I$.
\end{prop}

\smallskip

\section{Cycle sets of finite primitive level}

After introducing the class of cycle sets of finite primitive level, in this section we give a characterization of cycle sets of finite primitive level by its associated permutation group. 

\medskip

We start by definition of cycle sets of finite primitive level, given the first time in \cite{cedo2021constructing} in terms of solutions.

\begin{defin}
A cycle set $X$ is said to be \emph{primitive} if $\mathcal{G}\left(X\right)$ acts primitively on $X$. Moreover, we say that a finite indecomposable cycle set $X$ has \emph{primitive
level $k$}, and we will write $fpl(X)=k$, if $k$ is the biggest positive integer such that 
\begin{enumerate}
    \item[(1)] there exist cycle sets $X_1 = X, X_2, \ldots  ,X_k$, with $|X_i| > |X_{i+1}| > 1$, for every $1\leq  i\leq k-1$;
    \item[(2)] there exists an epimorphism of cycle sets $p_{i+1}: X_{i}\rightarrow X_{i+1}$, for every $1\leq  i\leq k-1$;
    \item[(3)] $X_k$ is primitive.
\end{enumerate}
\end{defin}

Clearly, every indecomposable cycle set having prime size is primitive and, by the main theorem of \cite{cedo2020primitive}, there are no other primitive cycle sets.

\begin{rem}\label{ossfixed}
As observed in \cite[Corollary 5.5]{castelli2022simplicity}, if $X$ is an indecomposable cycle set of finite primitive level and $x$ an arbitrary element of $X$, then $\sigma_x$ can not have a fixed point $y$, i.e., an element $y\in X$ such that $x\cdotp y=y$. However, this condition is not sufficient (see comment after \cite[Corollary 5.7]{castelli2022simplicity}).
\end{rem}

The following lemma is a simple but useful result to state when an indecomposable cycle set has finite primitive level.

\begin{lemma}\label{crit}
    Let $X$ be an indecomposable cycle set. Then, $X$ has finite primitive level if and only if there exist a trivial indecomposable cycle set $Y$ and an epimorphism $p:X\longrightarrow Y$. 
\end{lemma}

\begin{proof}
    Straightforward.
\end{proof}


Before giving the main result of this section, we give a preliminary lemma. This result is essentially a mixture between  \cite[Proposition 4.3]{cedo2021constructing} and the results contained in \cite[Section 1]{Ru220}.

\begin{lemma}\label{preldis}
    Let $X$ be an indecomposable cycle set. Then, the ideal $\mathcal{G}(X)^2$ of the left brace $\mathcal{G}(X)$ is equal to $Dis(X)$. Moreover, the factor group $\mathcal{G}(X)/Dis(X)$, regarded as left brace, is a trivial left brace with cyclic multiplicative (and hence additive) group.
\end{lemma}


Now we are able to show the main result of the section.

\begin{theor}\label{carattfinprim}
    Let $X$ be an indecomposable cycle set. Then, $X$ has finite primitive level if and only if $Dis(X)$ does not acts transitively on $X$.
\end{theor}

\begin{proof}
     Suppose that $X$ has finite primitive level. Then by \cref{dyn} $X$ is isomorphic to a dynamical extension $I\times_{\alpha} S$, where $I$ is an indecomposable cycle set having prime size $p$. Moreover by \cite[Theorem 2.13]{etingof1998set}, $I$ is a trivial cycle set. These facts implies that $Dis(X)$ fixes the subsets $\{i\}\times S$, for every $i\in I$, hence it can not acts transitively on $X$.\\
    Conversely, suppose that $Dis(X)$ does not acts transitively on $X$ and let $\Delta:=\{\Delta_1,...,\Delta_m\}$ the set of its orbits. Since by \cref{preldis} $Dis(X)$ is a normal subgroup of $\mathcal{G}(X)$, it follows that $ \mathcal{G}(X)$ acts on $\Delta$. Moreover, we have that  $\sigma_xDis(X)=\sigma_yDis(X)$ for every $x,y\in X$, therefore by \cref{preldis} $\mathcal{G}(X)/Dis(X)$ is a cyclic group generated by an element $\sigma_xDis(X)$. Now, since $\mathcal{G}(X)$ acts transitively on $X$, we have that every $\sigma_x$ acts on $\Delta$ as a cycle $\delta$ of lenght $m$. Therefore, the map $r:X\longrightarrow \Delta$, $x\mapsto \Delta_x$, where $\Delta_x$ is such that $x\in \Delta_x$, is an epimorphism from $X$ to the trivial cycle set $(\Delta,\cdotp)$ given by $\Delta_i\cdotp \Delta_j:=\delta(\Delta_j)$ for every $\Delta_i,\Delta_j\in \Delta$, hence the thesis follows by \cref{crit}.
\end{proof}


As a corollary, we provide a necessary condition to test the simplicity of a cycle set. Recall that an indecomposable cycle set $X$ is said to be \emph{simple} if it has no epimorphic images different from itself and the singleton.

\begin{cor}
    Let $X$ be an indecomposable simple cycle set such that $|X|$ is not a prime number. Then, $Dis(X)$ acts transitively on $X$.
\end{cor}

\begin{proof}
    It follows directly by \cref{carattfinprim}.
\end{proof}

In the particular case of cycle sets having prime-squared size, the condition of the previous result also is sufficient.

\begin{cor}
    Let $X$ be an indecomposable cycle set having size $p^2$, for a prime number $p$. Then, $X$ is a simple cycle set if and only if $Dis(X)$ acts transitively on $X$.
\end{cor}

\begin{proof}
    Since $|X|=p^2$ for a prime number $p$, then by \cite[Lemma 1]{cacsp2018} and \cite[Theorem 2.13]{etingof1998set} $X$ is simple if and only if it has not finite primitive level, hence the thesis follows by \cref{carattfinprim}.
\end{proof}

\section{Examples and applications}

In this section, we exhibit several examples and non-examples of cycle sets having finite primitive level, computing the exact primitive level in some cases. As an application of the main result of the previous section, we enumerate the indecomposable cycle sets of finite primitive level having small size. Moreover, we study some relations between these cycle sets and other classes recently considered in other papers.

\medskip

Many examples of indecomposable cycle sets having finite primitive level appeared in literature. Below, we exhibit some of them.

\begin{exs}\label{esfin}
    \begin{itemize}
         \item[1)] If $X$ is an indecomposable cycle set with $\mathcal{G}(X)$ abelian and $|X|=p_1^{\alpha_1}...p_n^{\alpha_n}$, where $p_1,...,p_n$ are distinct prime numbers, then $fpl(X)=\alpha_1+...+\alpha_n$ (see \cite[Theorem 4.4]{castelli2022simplicityarx} for more details). These cycle sets have been explicitly classified if $mpl(X)=2$ (see \cite{JePiZa20x}) and if $\mathcal{G}(X)$ is cyclic (see \cite{jedlivcka2021cocyclic}).
        \item[2)] Every indecomposable cycle set $X$ having finite multipermutation level is a cycle set of finite primitive level (see \cite[Corollary 4.5]{castelli2022simplicityarx}) and $mpl(X)\leq fpl(X)$ (several concrete examples of indecomposable cycle sets belonging to the multipermutation ones are contained, for example, in \cite{jedlivcka2023indecomposable} and \cite{cedo2022indecomposable}). 
       \item[3)] Let $k$ be an odd number and $(S,\cdotp)$ be the trivial cycle set given by  $S:=\mathbb{Z}/k\mathbb{Z}$ and $x\cdotp y:=y+1$ for all $x,y\in S$ and $(I,\star)$ the cycle set given by $I:=\{1,2,3,4\}$, $\sigma_1:=(1\quad 4) $, $\sigma_2:=(1\quad 3\quad 4\quad 2) $, $\sigma_3:=(2\quad 3) $ and $\sigma_4:=(1\quad 2\quad 4\quad 3) $. Then, the direct product $S\times I$ is a cycle set of finite primitive level, since the projection on the first component $S\times I\longrightarrow S$, $(s,i)\mapsto s$ give rise to a cycle sets epimorphism. This family of cycle sets appears in \cite[Example $9$]{cacsp2018}. Note that these cycle sets are not of finite multipermutation level since $Ret(S\times I)\cong I$.
        \item[4)] Let $G=\mathbb{Z}/6\mathbb{Z}\times \mathbb{Z}/6\mathbb{Z} $ and $\cdotp$ the binary operation on $G$ given by $$(i,j)\cdotp (k,l):=(k-j,l+t_{k-i}) $$
        where $t_x=1$ if $x=0$ and $t_x=3$ otherwise. Then, $(G,\cdotp)$ is the indecomposable cycle set constructed in \cite[Remark 4.10]{cedo2021constructing}. By a standard calculation, one can show that $(G,\cdotp)$ has the indecomposable trivial cycle set of size $2$ as epimorphic image.
    \end{itemize}
\end{exs}

As a generalization of 3) of \cref{esfin}, one can easily show that the class of indecomposable cycle sets of finite primitive level is closed by indecomposable dynamical extension.

\begin{prop}\label{dynfinite}
    Let $I$ be an indecomposable cycle set of finite primitive level and let $I\times_{\alpha} S$ be an indecomposable dynamical extension. Then, $I\times_{\alpha} S$ is an indecomposable cycle set of finite primitive level.
\end{prop}

\begin{proof}
    Straightforward.
\end{proof}

By previous proposition, examples of indecomposable cycle sets of finite primitive level occur in abundance (see \cite[Section 5]{cacsp2018} for several concrete examples).

\smallskip

In 1) of \cref{esfin} we exhibit a family of cycle sets for which the primitive level assumes the maximum possible value. In the next results, we show that this also happens for indecomposable cycle sets of square-free size and for the ones having cyclic permutation left brace. 

\begin{theor}\label{sqfree}
    Let $X$ be an indecomposable cycle set having size $p_1 ... p_n$, where $n$ is a natural number and $p_1,...,p_n$ are distinct prime numbers. Then, $fpl(X)=n$.
\end{theor}

\begin{proof}
    We show the thesis by induction on $n$. If $n=1$, the thesis directly follows by \cite[Theorem 2.13]{etingof1998set}. Now, suppose that $X$ is an indecomposable cycle set having size $p_1 ... p_n$, where $n$ is a natural number and $p_1,...,p_n$ are distinct prime numbers. By \cite[Theorem 4.1]{cedo2022indecomposable}, we have that $\mathcal{G}(X)=P_1\circ ...\circ P_n$, where $P_i$ is the $p_i$-Sylow of $(\mathcal{G}(X),+)$, and without loss of generality we can suppose that $P_1\circ ...\circ P_i$ is an ideal of $\mathcal{G}(X)$ for all $i\in \{1,...,n\}$. In particular, $P_1$ is an ideal of $\mathcal{G}(X)$ and hence a normal subgroup. Therefore, the orbits of $P_1$ form an imprimitive blocks system of $X$, hence necessarily every orbit of $P_1$ must have size $p_1$. By \cref{epiid}, $P_1$ induces a cycle set structure $X/P_1$ of size $p_2...p_n$ and the natural map from $X$ to $X/P_1$ is a cycle set epimorphism. By inductive hypothesis, it follows that $fpl(X/P_1)=n-1$, and since $X/P_1$ is an epimorphic image of $X$ the thesis follows.
\end{proof}

\begin{theor}\label{permcyc}
    Let $X$ be an indecomposable cycle set having size $p_1^{\alpha_1} ... p_n^{\alpha_n}$, where $n$ is a natural number and $p_1,...,p_n$ are distinct prime numbers. Moreover, suppose that $\mathcal{G}(X)$ is a permutation left brace with cyclic additive group. Then, $fpl(X)=\alpha_1+...+\alpha_n$.
\end{theor}

\begin{proof}
    We show the thesis by induction on $\alpha_1+...+\alpha_n$. If $\alpha_1+...+\alpha_n=1$ necessarily we have $n=1$ and $\alpha_1=1$, therefore the thesis follows by \cite[Theorem 2.13]{etingof1998set}. Now, suppose that $\alpha_1+...+\alpha_n>1$. Then by \cite[Corollary of Proposition 14]{rump2019classification}, since $\mathcal{G}(X)$ has cyclic additive group, we have that $|Soc(\mathcal{G}(X))|>1$, therefore there exist a normal subgroup $(I,+)$ of $(Soc(\mathcal{G}(X)),+)$ having prime size. Without loss of generality, we can suppose that $|I|=p_1$. Since $I$ is a characteristic subgroup of $(Soc(\mathcal{G}(X)),+)$, it follows that $I$ is an ideal of $ \mathcal{G}(X)$. Moreover, every orbit of $X$ respect to the action of $I$ must have size $p_1$ and hence it induces a cycle set $X/I$ of size $p_1^{\alpha_1-1}...p_n^{\alpha_n}$. By inductive hypothesis, we have that $fpl(X/I)=\alpha_1-1+...+\alpha_n$ and hence $fpl(X)\geq fpl(X/I)+1=\alpha_1+...+\alpha_n $, therefore the thesis follows.  
\end{proof}

\medskip

In \cite[Corollary 5.5]{castelli2022simplicity}, it has been shown that if $X$ is an indecomposable cycle set having finite primitive level, then every left multiplication has not fixed points. Here, we give a further information involving the cycle decomposition of the left multiplications. In the following, a $k$-cycle $(x_1\dots x_k)$ will be called \emph{trivial} if $k=1$.

\begin{prop}\label{divfpl}
    Let $X$ be an indecomposable cycle set of finite primitive level and $\{\alpha_1,...,\alpha_n\}$ the set of all the cycles (possibly trivial) belonging to at least a left multiplication $\sigma_x$. Then, there exist a prime divisor $p$ of $|X|$ that divides the length of $\alpha_i$, for all $i\in \{1,...,n\}$.
\end{prop}

\begin{proof}
    Since $X$ has finite primitive level, by \cref{dyn} there exist a prime number $p_1$ such that $X$ is isomorphic to a dynamical extension $I\times_{\alpha} S $ and $I$ is an indecomposable cycle set of size $p_1$, which by \cite[Theorem 2.13]{etingof1998set} can be identified with the one given by $I:=\mathbb{Z}/p_1\mathbb{Z}$ and $x\cdotp y:=y+1$ for all $x,y\in I$. Now, let $(i,s)$ and $(j,t)$ be elements of $I\times S$. If $(j,t)$ belongs to a $z$-cycle of $\sigma_{(i,s)} $ (by \cite[Corollary 5.5]{castelli2022simplicity} we must have $z>1$), then $(j,t)=\sigma^z_{(i,s)}(j,t)=(j+z,t)$, therefore $p_1$ must divides $z$. Since $(i,s)$ and $(j,t)$ are arbitrary elements of $X$, $p_1$ is the desired prime number.  
\end{proof}

\noindent In \cite{RaVe21}, Ramirez and Vendramin posed the following question.

\begin{ques}\label{quesven}
    Let $X$ be a cycle set. Is it true that if some $\sigma_x$ contains a non-trivial cycle of length coprime with $|X|$, then $X$ is decomposable?
\end{ques}

As an application of \cref{divfpl}, we give a positive answer when $X$ has finite multipermutation level.

\begin{cor}\label{corques}
     Let $X$ be a multipermutation cycle set. Suppose that some $\sigma_x$ contains a cycle of length coprime with $|X|$. Then, $X$ is decomposable.
\end{cor}

\begin{proof}
    If suppose $X$ indecomposable, then it has finite primitive level. Then, by \cref{divfpl} every cycle contained in an arbitrary $\sigma_x$ has not coprime length with $|X|$, a contradiction.
\end{proof}

Actually, we are not able to state if the hypothesis on the multipermutation level can be dropped in \cref{corques}. In this context, note that a possible counterexample $X$ to \cref{quesven} would imply that $\mathcal{G}(X) $ is a \emph{singular} left brace, where a left brace $B$ is said to be singular if there exist and an indecomposable cycle set $X$ such that $B\cong \mathcal{G}(X)$ and a prime number $p$ that divides the order of $B$ but not the order of $X$ (for this reason, these cycle sets also will called singular). These left braces are recently characterised in \cite{rump2023primes}. Singular cycle sets seem to be very difficult to construct: actually, only a counterexample of size $8$, given in \cite{rump2023primes}, is known in literature. In the same paper, it was also shown that if $X$ is a singular cycle set, then so is its retraction, therefore the research of these cycle sets can be reduced in some sense with the irretractable ones. \\
Below, we recall Rump's singular cycle set and we use \cref{carattfinprim} to show that it has finite primitive level. Moreover, we use this cycle set to construct, by dynamical extensions obtained in \cite[Section 5]{cacsp2018}, a family of irretractable singular cycle sets.

\begin{ex}
Let $X:=\{0,1,2,3,4,5,6,7\}$ be the indecomposable cycle set given by 
$$ \sigma_0  = (07)(13)(25)(46) \qquad  \sigma_1 = (0264)(1375) $$
    $$\sigma_2 = (01)(25)(34)(67) \qquad \sigma_4 = (0462)(1573) $$
$$\sigma_3 = (02)(16)(34)(57) \qquad \sigma_5 = (0451)(2673) $$
$$\sigma_7 = (07)(16)(23)(45)\qquad  \sigma_6 = (0154)(2376).
 $$
 Then, the left brace $\mathcal{G}(X)$, that has size $24$, is singular since $3$ divides $|\mathcal{G}(X)|$ but not $|X|$. By a standard calculation, one can show that $\mathcal{G}(X)^2$ splits $X$ into the orbits $\{0,3,5,6\}$ and $\{1,2,4,7 \}$, hence by \cref{carattfinprim} $X$ has finite primitive level.  
  Now, let $S:=\mathbb{Z}/k\mathbb{Z}$, with $k$ an arbitrary number coprime with $3$, $A$ the set given by $A:=S\times S$, and $\alpha$ be the function from $X\times X\times A$ to $\Sym(A)$, $\alpha(x,y,(a,b))\mapsto\alpha_{(x,y)}((a,b),-)$ given by
 $$
\alpha_{(x,y)}((a,b),(c,d)):= \begin{cases} (c,d+1) & \mbox{ if }x=y \mbox{ and } a\neq c\\
(c,d) & \mbox{ if }x=y \mbox{ and } a= c\\
(c-b-1,d) & \mbox{ if }x\neq y 
\end{cases}
 $$
 for all $(x,y,(a,b))\in X\times X\times A$. By a standard calculation, one can show that the dynamical extension $X\times_{\alpha} A$ is an indecomposable cycle set and by \cite[Proposition 10]{cacsp2018} is irretractable. Since $\mathcal{G}(X)\cong \mathcal{G}(X\times_{\alpha} A)/I$ for a suitable ideal $I$, we have that $3$ divides $|\mathcal{G}(X\times_{\alpha} A)|$; on the other hand, $3$ does not divide $|X\times A|$. Moreover, by \cref{dynfinite}, the cycle set $X\times_{\alpha} A$ is of finite primitive level.
\end{ex}

\noindent In this context, an intriguing challenge is the construction of further singular cycle sets that are in some sense different from the previous ones: for example, one could ask if there exist singular cycle sets which have not finite primitive level.

\medskip

For a natural number $n$, let $c(n)$ be the number of indecomposable cycle sets of size $n$, $m(n)$ be the number of indecomposable cycle sets of size $n$ having finite multipermutation level and $fp(n)$ be the number of indecomposable cycle sets of size $n$ having finite primitive level. 
As an application of \cref{carattfinprim}, by means of the GAP package \cite{Ve15pack}, we computed, by a small GAP code, the first values of $fp(n)$. We summarize our calculations in the following table.

\vspace{4mm}

\begin{center}
 \begin{tabular}{|c|c|c|c|}
        \hline
        $n$ & $c(n)$ & $m(n)$  & $fp(n)$ \\
        \hline
        2 & 1 & 1 & 1 \\
        \hline
        3 & 1 & 1 & 1 \\
        \hline
        4 &  5 & 3 & 3 \\
        \hline
        5 & 1 & 1 & 1 \\
        \hline
        6 & 10 & 10 & 10 \\
        \hline
        7 & 1 & 1 & 1 \\
        \hline
        8 & 100 & 39 & 70 \\
        \hline
        9 & 16 & 13 & 13 \\
        \hline
    \end{tabular}
\end{center}

\begin{rem}
 For every $n\in \{2,...,9\}$ we have $m(n)\leq fp(n)$ and, if $n$ is a prime number, we obtain $m(n)=fp(n)=1$: these facts agree with 2) of \cref{esfin} and \cite[Theorem 2.13]{etingof1998set}. If $n=6$ we have $c(n)=m(n)=fp(n)$: this is consistent with \cite[Theorem 4.5]{cedo2022indecomposable} and \cref{sqfree}.    
\end{rem} 

\medskip

In the last part of this section, we focus on some classes of cycle sets present in literature that provide examples of indecomposable cycle sets which have not finite primitive level. 

\begin{exs}
    \begin{itemize}
        \item[1)] Every non-trivial simple cycle set has not finite primitive level because the only epimorphic images are the whole cycle set and the cycle set of size one (see \cite{castelli2022characterization,cedo2021constructing,cedo2022new} for several concrete examples).
        \item[2)] Let $X:=\{1,2,3,4,5,6,7,8\}$ and $\cdotp$ the binary operation given by $$\sigma_1=\sigma_2:=(3\quad 5\quad 4\quad 7)$$
        $$\sigma_3=\sigma_4:=(1\quad 6\quad 2\quad 8)$$
        $$\sigma_5=\sigma_7:=(1\quad 5\quad 6\quad 4\quad 2\quad 7\quad 8\quad 3)$$ 
        $$\sigma_6=\sigma_8:=(1\quad 3\quad 6\quad 5\quad 2\quad 4\quad 8\quad 7).$$
        Then, $2$ is a fixed point of $\sigma_1$ hence by \cref{ossfixed} $X$ can not have finite primitive level. This cycle set was given in \cite[Example 3.8]{acri2020retractability}. 
        Inspecting \cite[Table 3.2 and 3.3]{acri2020retractability} and by means of \cref{ossfixed}, one can find further examples of indecomposable cycle sets that are not of finite primitive level. 
    \end{itemize}
\end{exs}

A large family of indecomposable cycle sets that are not of finite primitive level is given by the so-called \emph{latin} cycle sets, where a cycle set is said to be latin if the right multiplication $\delta_x:X\longrightarrow X$, $y\mapsto y\cdotp x$ is bijective, for every $x\in X$ (see \cite{bon2019} for more details and concrete examples). Clearly, these cycle sets always are indecomposable.

\begin{cor}\label{quasigroupprim}
    Let $X$ be a latin cycle set, with $|X|>1$. Then, $X$ has not finite primitive level.
\end{cor}

\begin{proof}
If $x,y\in X$, if $z$ is an other element of $X$, there exist $t\in X$ such that $y=t\cdotp (\sigma_z^{-1}(x))=\sigma_t(\sigma_z^{-1}(x))$, hence $Dis(X)$ acts transitively on $X$. Therefore the thesis follows by \cref{carattfinprim}.
\end{proof}



In \cite{ballester2023solubility} the notion of \emph{soluble} solution (not necessarily involutive) was recently introduced. Here, we recall such a notion, restricting to an involutive setting and using the language of cycle sets, and we close the section showing that this class of cycle sets has empty intersection with the one of the cycle sets of finite primitive level.


\begin{defin}\label{soluble}
    Let $X$ be a cycle set. Assume that there exists a sequence of subsets $X_t\subseteq...\subseteq X_1 \subseteq  X_0 = X$ with $X_t = \{x_t\}$ such that, for every $1 \leq i \leq t$, there exist a a cycle set $Y_i$ and a cycle set epimorphism $f_i:X\longrightarrow Y_i$ satisfying
    \begin{itemize}
        \item[1)] $X_i \in  X/Ker_{f_i}
$ for all $ 1 \leq i \leq t$;
        \item[2)]  $f_i(X_{i-1})$ is a trivial sub-cycle set of $Y_i$ given by $x\cdotp y=y$ for every $x,y\in f_i(X_{i-1})$, for all $1\leq i\leq t$.
    \end{itemize}
Then, $X$ is said to be \emph{soluble at $x_t$}.
\end{defin}


\begin{prop}
    Let $X$ be an indecomposable cycle set having finite primitive level. Then, $X$ is not a soluble cycle set.
\end{prop}

\begin{proof}
    Suppose that $X$ is soluble and let $X_0,...,X_t$ and $f_1,...,f_t$ as in \cref{soluble}. Then, we have that $X_t\in X/Ker_{f_t}$, and by \cite[Lemma 1]{cacsp2018} we have that $f_t$ is bijective, therefore $X\cong Y_t$. By 2) of \cref{soluble}, there exist $x,y\in Y_t$ such that $x\cdotp y=y$, but this contradicts \cref{ossfixed}.
\end{proof}

\section{Cycle sets of primitive level $2$}

In this section, we focus on cycle sets having primitive level $2$. In particular, following \cite[Question 3.2]{cedo2021constructing}, we provide a description of all the indecomposable cycle sets having primitive level $2$ by means of their permutation left braces.

\smallskip

At first, we start by an easy case, considering cycle sets with trivial permutation left brace.

\begin{prop}
    Let $X$ be an indecomposable cycle set with trivial permutation left brace $\mathcal{G}(X)$. Then, $fpl(X)=2$ if and only if $X$ has size $pq$, where $p$ and $q$ are two prime numbers not necessarily distinct.
\end{prop}

\begin{proof}
    Since $\mathcal{G}(X)$ is a trivial left brace, it follows that $X$ is a trivial cycle set and any epimorphic image of $X$ is a trivial cycle set. Then, the thesis follows by the fact that if $Y$ is a trivial indecomposable cycle set and the size of $Y$ divides the size of $X$, then $Y$ is an epimorphic image of $X$. 
\end{proof}

By the previous proposition, we can focus on 
indecomposable cycle set of primitive level $2$ provided by non-trivial left braces. We start by some preliminary results. 

\begin{prop}\label{cardtrivdiv}
    Let $X$ be an indecomposable cycle set and $p:X\longrightarrow Y$ an epimorphism from $X$ to a trivial indecomposable cycle set. Then, the size of $Y$ divides $X/\mathcal{G}(X)^2$.
\end{prop}

\begin{proof}
    By \cref{dyn}, $X$ is isomorphic to a dynamical extension $Y\times_{\alpha} S$. Moreover, $\mathcal{G}(X)^2$ fixes every set $\{y\}\times S$, for all $y\in Y$, and by a standard calculation we have that $x_1$ and $x_2$ are in the same orbit respect to $\mathcal{G}(X)^2 $ if an only if $g(x_1)$ and $g(x_2)$ are in the same orbit respect to $\mathcal{G}(X)^2 $, for all $x_1,x_2\in X$ and $g\in \mathcal{G}(X)$. Therefore, there exist a natural number $r$ such that $\mathcal{G}(X)^2 $ splits every set $\{y\}\times S$ into $r$ orbits. Hemce, it follows that $|X/\mathcal{G}(X)^2 |=r\cdotp |Y|$.
\end{proof}

\begin{cor}\label{cardtriv}
    Let $X$ be a non-trivial indecomposable cycle set having primitive level $2$. Then, the action of $\mathcal{G}(X)^2$ on $X$ splits $X$ into $p$ orbits, for a prime number $p$, and if $r:X\longrightarrow Y$ is an epimorphism with $|Y|$ a prime number, then $|Y|=p$.
\end{cor}

\begin{proof}
   Since $X$ has primitive level $2$, necessarily the action of $\mathcal{G}(X)^2$ on $X$ splits $X$ into $p$ orbits, for a prime number $p$. Since $|Y|$ is a prime number, the thesis follows by the previous corollary.
\end{proof}

Now, we are ready for the desired description.

\begin{theor}\label{liv2}
    Let $(B,+,\circ)$ be a non-trivial left brace, $Y\subset B$ a transitive cycle base, $a_1\in Y$ and $K$ a core-free subgroup contained in the stabilizer $B_{a_1}$ of $a_1$ respect to the action $\lambda$.
    Moreover, let $x\circ K$ be an arbitrary left coset of $B/K$ and $B_{x\circ K}$ be the stabilizer of $x\circ K$ in $B$ respect to the left multiplication in $(B,\circ)$.
    Then, the cycle set $C_{B,K,a_1}$ has primitive level $2$
    if and only if the following conditions hold:
    \begin{itemize}
        \item[1)] the index of the subgroup $B^2\circ B_{x\circ K} $ of $B$ is a prime number $p$;        
        \item[2)] the action of $B^2$ on the left coset $B/H$ by left multiplication is transitive, for every core-free subgroup $H$ with $K<H\leq B_{a_1}$;
        \item[3)] if $J$ is an ideal such that its action on 
        the left coset $B/K$ by left multiplication has $o_J$ orbits, with $o_J>p$, then $B^2$ acts transitively (by the induced action) on the $J-$orbits of $B/K$.
    \end{itemize}
Moreover, every non-trivial indecomposable cycle set $X$ having primitive level $2$ can be constructed as $C_{B,K,a_1}$ for suitable $B,K$ and $a_1$ satisfying the previous conditions.
\end{theor}

\begin{proof}
    Suppose that $C_{B,K,a_1}$ has primitive level $2$. Then, $C_{B,K,a_1}/B^2$ is a non-trivial quotient of $C_{B,K,a_1}$ and the size of  $C_{B,K,a_1}/B^2$ is a prime number $p$.
    Since the action of $B^2$ on the cycle set $C_{B,K,a_1}$ is just the action by left multiplication on $B/K$, by \cite[Excercise 9 at p. 117]{dummitfoote} we have that condition 1) follows. If 2) does not hold, there exist a core-free subgroup $H$, with $K<H$, that give rise to a covering $p_1:C_{B,K,a_1}\longrightarrow C_{B,H,a_1} $ and an epimorphism $p_1:C_{B,H,a_1}\longrightarrow C_{B,H,a_1}/B^2 $ with $ |B/H| < |B/K| $ and $|C_{B,H,a_1}/B^2 |>1$, therefore by \cref{carattfinprim} $C_{B,H,a_1}$ has finite primitive level and hence $C_{B,K,a_1}$ has primitive level greater that $2$, a contradiction. If 3) does not hold for a suitable ideal $J$, we obtain an epimorphism $p_1:C_{B,K,a_1}\longrightarrow C_{B,K,a_1}/J $ and, if $J'$ is such that $\mathcal{G}(C_{B,K,a_1}/J)\cong B/J'$, we have that $(B/J')^2$ does not acts transitively on $C_{B,K,a_1}/J$. Therefore, by \cref{carattfinprim} $C_{B,K,a_1}/J $ has finite primitive level and hence $C_{B,K,a_1} $ has primitive level greater than $2$, a contradiction.\\
    Conversely, suppose that 1), 2) and 3) hold.  By condition 1) and \cite[Excercise 9 at p. 117]{dummitfoote}, the indecomposable cycle set $Z:=C_{B,K,a_1}/B^2$ has prime size $p$. Moreover, there is a natural epimorphism $r$ from $C_{B,K,a_1}$ to $Z$.
    By \cref{cardtrivdiv} and \cite[Theorem 2.13]{etingof1998set}, $p$ and $Z$ are completely determined by condition 1) and there are not other trivial indecomposable cycle sets that are epimorphic images of $C_{B,K,a_1}$. Therefore, to show the thesis, it is sufficient proving that there is not a non-trivial indecomposable cycle set $T$, different from $C_{B,K,a_1}$ and $Z$, such that $T$ is an epimorphic image of $C_{B,K,a_1}$ and $Z$ is an epimorphic image of $T$. Suppose $T$ is such a cycle set and $r_1:C_{B,H,a_1}\longrightarrow T$ and $r_2:T\longrightarrow Z$ be epimorphisms. If $r_1$ is a covering, by \cref{coveringcostr} $T=B/H$ for some subgroup $H$ with $K<H\leq B_{a_1}$, and since $T$ has finite primitive level, $B^2$ does not acts transitively on the left cosets $B/H$, against condition 2). Then, by \cref{fattepi}, without loss of generality we can suppose that $T$ is an epimorphic image of the form $C_{B,H,a_1}/J$, for some non-trivial ideal $J$. Therefore $C_{B,H,a_1}/J$ is a non-trivial indecomposable cycle set of finite primitive level, with $|C_{B,H,a_1}/J|=o_j>p$ and this implies that $B^2$ does not acts transitively on the $J$-orbits of $B/K$, but this contradicts 3). \\   
 Finally, by \cref{cosmod} every non trivial indecomposable cycle set $X$ having primitive level $2$ can be constructed as $C_{B,K,a_1}$ for suitable $B,K$ and $a_1$ satisfying conditions 1), 2) and 3).
\end{proof}


We conclude the section applying \cref{liv2} to construct a family of indecomposable cycle set having primitive level $2$.

\begin{ex}
    Let $B_1$ be the left brace $B_{8,27}$ of \cite{Ve15pack} and $B_2$ the trivial left brace having $p$ elements, for a prime number $p$ different from $2$, and set $B$ the direct product of the left braces $B_1$ and $B_2$. Then, $B$ has a transitive cycle base $Y=Y_1\times\{y\}$, where $Y_1$ is a transitive cycle base of $B_1$, which has size $4$, and $y$ is a non-zero element of $B_2$. Moreover, every element $a$ of $Y_1$ is stabilized by a core-free subgroup $K_a'$ of $(B_1,\circ)$ having size $2$. Therefore, if we set $a_1\in Y$ and  $K:=K_{a_1}'\times \{0\}$, we obtain that $C_{B, K, a_1}$ is an indecomposable cycle set having size $4p$. Now, we show that it is of primitive level $2$. If $x\circ K$ is a left coset of $B/K$, we obtain that $B^2\circ B_{x\circ K} $ is equal to $B_1\times \{0\}$, which is a subgroup of $(B,\circ)$ of index $p$, therefore condition 1) of \cref{liv2} is satisfied. Since $(B_1,\circ)$ is the dihedral group of size $8$ and $(B_2,\circ)$ is cyclic of prime order $p\neq 2$, condition 2) of \cref{liv2} automatically follows. Finally, the ideals of $B$ different from $\{0\}$ and $B$ are: $B_1\times \{0\}$, $\{0\}\times B_2$, $B_1^{2}\times \{0\}$,  $B_1^{2}\times B_2$. We do not need to consider $B_1^{2}\times B_2$, since it acts transitively on $C_{B, K, a_1}$. The ideals $B_1\times \{0\}$ and $B_1^2\times \{0\}$ splits $C_{B, K, a_1}$ in the same way into $p$ orbits, hence the remaining case is the ideal $\{0\}\times B_2$. It splits $C_{B, K, a_1}$ into $4$ orbits and $B^2=B_1^2\times \{0\}$ acts transitively on these orbits, therefore condition 3) of \cref{liv2} also follows and hence $C_{B, K, a_1}$ has primitive level $2$.
\end{ex}

\section*{Acknowledgements}
The author thanks M. Bonatto for the discussion about latin cycle sets.

\bibliographystyle{elsart-num-sort}
\bibliography{Bibliography2}

\end{document}